\documentclass[12pt]{amsart}
\usepackage{amsfonts}
\usepackage{amssymb}
\usepackage{amsthm}
\usepackage{dsfont}
\usepackage{color}
\usepackage{url}
\usepackage{enumerate}

\newtheorem{theorem}{Theorem}[section]
\newtheorem{lemma}[theorem]{Lemma}

\newtheorem{corollary}[theorem]{Corollary}

\newtheorem{fact}[theorem]{Fact}

\newtheorem{problem}[theorem]{Problem}

\newtheorem{remark}[theorem]{Remark}

\newcommand{\I}{\mathcal I}

\newcommand{\N}{\mathbb N}
\newcommand{\Q}{\mathbb Q}
\newcommand{\R}{\mathbb R}

\newcommand{\ve}{\varepsilon}

\newcommand{\mc}{\mathcal}
\newcommand{\sbq}{\subseteq}
\newcommand{\wf}{\widetilde{f}}

\author{Marek Balcerzak}
\address{Institute of Mathematics,
         Lodz University of Technology, al. Politechniki 8,
         93-590 \L\'od\'z,
          Poland}
         \email{marek.balcerzak@p.lodz.pl}
         
\author{Micha{\l} Pop{\l}awski}
\address{Institute of Mathematics,
         Lodz University of Technology, al. Politechniki 8,
         93-590 \L\'od\'z,
          Poland}
         \email{michal.poplawski@p.lodz.pl}
         
 \title{Extending theorems of Boros and Menzer}        
\begin{document}
\begin{abstract}
We extend results of Boros and Menzer on the alternative equation $f(x)f(y)=0$ for generalized polynomials $f$, and their theorems on the conditional inequality $f(x)f(y)\ge 0$ for generalized monomials $f$ of even degree. We use similar methods and ideas. We replace the largeness, of the respective Borel plane set $D$, in the measure or in the Baire category sense, by its largeness in the mixed measure-category sense.
\end{abstract}
\maketitle
\section{Introduction}
In the recent papers \cite{BM1} and \cite{BM}, Boros and Menzer considered the alternative function equation
$f(x)f(y)=0$ for generalized polynomials $f$, and the conditional function inequality $f(x)f(y)\ge 0$ for generalized monomials of even degree $m\ge 2$, under measure or Baire category constraints on $f$. 
In fact, their theorems for alternative equations are formulated in a general form where $f\colon G\to \mathbb C$ and $G$ is a locally compact Abelian group generated by a neighbourhood of zero, with a $\sigma$-finite Haar measure, and in the category case, $f\colon\R^N\to\mathbb C$. For the conditional inequality, Boros and Menzer assumed simply that $f\colon\R\to\mathbb R$. Our main purpose is to extend results of \cite{BM} to the setting where one assumes mixed measure-category constraints on $f$.

Let us recall from \cite{BM} some basic facts concerning generalized monomials and polynomials. See also \cite{Ku}, \cite{Sz}.
Let $G$ be an Abelian topological group with the operation $+$. Given $f\colon G\to\mathbb C$ and $h,x\in G$, we define $\Delta_h f(x):=f(x+h)-f(x)$. If $m\in\N$ and $F\colon G^m\to\mathbb C$ is a non-zero symmetric $m$-additive function (that is, a function addtitive in each variable), consider its diagonalization
\begin{equation*} \label{EE0}
f(x):=F(\overbrace{x,x,\dots,x}^m)\quad (x\in G).
\end{equation*} 
Then $f$ is called a \emph{generalized monomial} of degree $m$. Additionally, we may call constant functions as generalized monomials of degree zero. Note that $f(rx)=r^mf(x)$ for all $r\in\mathbb Q$ and $x\in G$ when $f$ is of positive degree $m$.
If $F$ and $f$ are as above, we have
\begin{equation*} 
\Delta_{h_1}\Delta_{h_2}\dots\Delta_{h_m}f(x)=m!\, F(h_1,h_2,\dots,h_m)
\end{equation*}
for all $x,h_1,h_2,\dots ,h_m\in G$.
Moreover, it follows that
\begin{equation*} 
\Delta_{h_1}\Delta_{h_2}\dots\Delta_{h_n}f(x)=0
\end{equation*}
for all $x,h_1,h_2,\dots ,h_n\in G$ if $n$ is an integer such that $n>m$.
Assuming that $G$ is uniquely divisible by $(m+1)!$, a \emph{generalized monomial}
 $f\colon G\to\mathbb C$ of degree $m$ can be characterized as a solution of
the functional equation
\begin{equation} \label{EE1}
\Delta_y^m f(x)=m!\, f(y)\quad (x,y\in G)
\end{equation}
where $\Delta_y^m$ denotes the $m$-th iterate of the operation $\Delta_y$.

A finite sum of generalized monomials is called a \emph{generalized polynomial}. More exactly,
a generalized polynomial of degree at most $p\in\N$ is a function $f\colon G\to \mathbb C$ of the
form
\begin{equation} \label{EE2}
f(x)=\sum_{k=0}^p f_k(x)\quad (x\in G)
\end{equation}
where functions $f_k$ (for $k=1,\dots, p$) are associated with symmetric $k$-additive functions $F_k$ as follows
$$f_k(x):=F_k(\overbrace{x,x,\dots,x}^k)\quad (x\in G)$$
and $f_0(x):=F_0\in\mathbb C$ for all $x\in G$.
In fact, generalized polynomials of degree at most $p$ can be characterized as solutions of the following functional equation, provided that $G$ is uniquely divisible by $(p+1)!$:
\begin{equation} \label{EE3}
\Delta^{p+1}_y f(x)=0\quad (x,y\in G).
\end{equation}

Through the paper, we will assume that $G=\R$ (or $G=\R^N$). It is known that, by the Axiom of Choice, there exist generalized discontinuous monomials (polynomials) $f\colon\R\to\mathbb C$. Note that a more restrictive notion of a generalized polynomial $f\colon G\to\mathbb C$ was considered in the recent paper \cite{KL}.

Let us state two pairs of results from \cite{BM} which have a form of an implication. They will be extended in this paper to a new situation in the sense that the conclusions will be the same, but our new assumptions will present completely different cases that cannot be derived from the old ones.
However, an idea of proofs remains the same.

The first theorem is quoted for the particular case where the considered group is $\R$.

\begin{theorem}(See \cite[Theorems 3.2 and 3.6]{BM}) \label{alt}
Let $f\colon\R\to\mathbb C$ be a generalized polynomial such that $f(x)f(y)=0$ for all $(x,y)\in D$ where $D\sbq\R^2$. Assume that one of the conditions is true:
\begin{itemize}
\item[(I)] $D$ is a Lebesgue measurable set of positive measure,
\item[(II)] $D$ is a Baire measurable nonmeager set.
\end{itemize}
Then $f(x)=0$ for each $x\in\R$.
\end{theorem}

Observe that the set $D$ in (I) and (II) can be assumed to be Borel. Indeed, it follows from the fact that a Lebesgue measurable set of positive measure clearly contains a closed set of positive measure. In the Baire category case, we note that a Baire measurable set is the union of a meager set and a $G_\delta$ set (cf. \cite[Theorem 4.4]{Ox}). So, if $f$ is a generalized polynomial and the function $(x,y)\mapsto f(x)f(y)$ vanishes on a Borel set $D\sbq\R^2$ that is large in the measure or the Baire category sense, then $f$ vanishes everywhere.

The next theorem is formulated in a similar manner.

\begin{theorem}(See \cite[Theorems 4.3 and 4.6]{BM}) \label{cond}
Let $f\colon\R\to\R$ be a generalized monomial of even degree $m\ge 2$ fulfilling $f(x)f(y)\ge0$ for all 
$(x,y)\in D$ where $D\sbq\R^2$. Assume that one of the conditions is true:
\begin{itemize}
\item[(I)] $D$ is a Lebesgue measurable set of positive measure,
\item[(II)] $D$ is a Baire measurable nonmeager set.
\end{itemize}
Then $f(x)f(y)\ge 0$ for each $(x,y)\in\R^2$.
\end{theorem}

As before, the set $D$ in (I) and (II) can be assumed to be Borel. So, if $f$ is a generalized monomial of even degree $m\ge 0$ and the function $\wf (x,y):=f(x)f(y)$ is non-negative on a Borel set $D\sbq\R^2$ that is large in the measure or the Baire category sense, then $\wf$ is non-negative everywhere.

We are going to obtain the analogues of Theorems \ref{alt} and \ref{cond} where a Borel set $D\sbq\R^2$ is large in the mixed measure-category sense.

\section{The Fubini products of $\sigma$-ideals}
The mixed measure-category largeness that we use in this paper is connected with the Mendez 
$\sigma$-ideals $\mc M\otimes\mc N$ and
$\mc N\otimes\mc M$ on the plane $\R^2$. See \cite{M1}, \cite{M2} and \cite{BG}. Here $\mc M$ and $\mc N$ stand for the $\sigma$-ideals of meager sets and Lebesgue null sets in $\R$. 

First, let us recall the notion of the Fubini product of two $\sigma$-ideals. Let $\mc I$ and $\mc J$ be given $\sigma$-ideals of subsets of $\R$. For $B\subseteq\R^2$ and $x\in\R$, we let $B_x:=\{y\in\R\colon (x,y)\in B\}$.
Then we put 
$$\Phi_{\mc J}(B):=\{x\in\R\colon B_x\notin\mc J\} .$$ 
Define $\mc I\otimes \mc J$ as the family of all sets $A\subseteq\R^2$ such that there exists a Borel set $B\subseteq\R^2$ with $\Phi_{\mc J}(B)\in\mc I$. It is easy to check that $\mc I\otimes \mc J$ is a $\sigma$-ideal. By the Fubini theorem, the Kuratowski-Ulam theorem and their converses (see \cite{Ox}) it follows that $\mc N\otimes\mc N$ and $\mc M\otimes\mc M$ are, respectively, the $\sigma$-ideals of Lebesgue null sets and meager sets in the plane. If we do not use a Borel cover $B$ in the definition of $\mc I\otimes\mc J$, some pathological sets can appear in $\mc N\otimes\mc N$ and $\mc M\otimes\mc M$ (cf. \cite[Theorems 14.4, 15.5]{Ox}).

Consider the well-known decomposition $\R=A\cup B$ of the line into disjoint Borel sets $A\in\mc M$ and $B\in\mc N$ (see \cite{Ox}). Mendez in \cite{M1} observed that there are the following two useful decompositions of the plane
\begin{itemize}
\item $(A\times\R)\cup (B\times \R)=\R^2\;\mbox{ with }\; A\times\R\in\mc M\otimes\mc N\;\mbox{ and }\; B\times\R\in\mc N\otimes\mc M;$
\item $C\cup H=\R^2$ where 
$C:=(A\times B)\cup(B\times A),\;H:=(A\times A)\cup(B\times B)$,\\ 
and
$C\in(\mc M\otimes\mc M)\cap(\mc N\otimes\mc N),\; H\in (\mc M\otimes\mc N)
\cap(\mc N\otimes\mc M).$
\end{itemize}
From the former decomposition it follows that there is no inclusion between $\mc M\otimes\mc N$
and $\mc N\otimes\mc M$.
From the latter decomposition it follows that there is no inclusion between any pair of $\sigma$-ideals:
\begin{itemize} 
\item $\mc M\otimes\mc M$ and $\mc M\otimes \mc N$;
\item $\mc N\otimes\mc N$ and $\mc M\otimes \mc N$;
\item $\mc M\otimes\mc M$ and $\mc N\otimes \mc M$;
\item $\mc N\otimes\mc N$ and $\mc N\otimes \mc M$.
\end{itemize}

In our results, large sets are Borel sets that do not belong to $\mc M\otimes\mc N$ or $\mc N\otimes\mc M$.
The lack of the respective inclusions between $\sigma$-ideals shows that this sort of largeness cannot be deduced from the largeness in the sense of measure or the Baire category in $\R^2$.

\section{A theorem on the alernative equation}
In this section we consider an alternative equation
$$f(x)f(y)=0,$$
assumimng that $f\colon\R\to\mathbb C$ is a generalized polynomial.

The following facts on zeros of generalized polynomials are stated in \cite[Theorems 3.1 and 3.4]{BM} in a more general form.

\begin{theorem} \label{TT1}
Let $f\colon\R\to\mathbb C$ be a generalized polynomial.
\begin{itemize}
\item[(i)] If $f$ vanishes on a measurable set $E\subseteq\R$ of positive measure, then $f$ vanishes everywhere on $\R$.
\item[(ii)] If $f$ vanishes on a Baire nonmeager set $E\subseteq\R$, then $f$ vanishes everywhere on $\R$.
\end{itemize}
\end{theorem}

Statement (i) of the above theorem was proved by Sz\'ekelyhidi \cite[Theorem 3.3, p. 33]{Sz} in a general version for a Haar measure on a locally compact Abelian group. Statement (ii) was proved in \cite{BM} for the space $\R^N$. 
Below, for the reader's convenience, we reconstruct a full simple proof of (i) which is analogous to the argument for (ii) presented in \cite{BM}.

Let $\lambda$ stand for Lebesgue measure on $\R$. The following fact and corollary are known (in \cite{Sz} this is referred to a result of Kurepa \cite{Kur}). It is connected with the classical theorem of Steinhaus.

\begin{fact} \label{fa}
Let $A\sbq\R$ be a bounded measurable set of positive measure, and let $n$ be a positive integer.
Then
$$\lim_{h\to 0}\lambda\left(A\cap\bigcap_{k=1}^n(A+kh)\right)=\lambda(A).$$
\end{fact}
\begin{proof}
The assertion follows from the equivalent condition
$$\lim_{h\to 0}\lambda\left(A\setminus\bigcap_{k=1}^n(A+kh)\right)=0$$
which will be proved below.

Let $\ve>0$. Pick a bounded open set $U\supseteq A$ and a closed nonempty set $F\sbq A$ such that 
\begin{itemize}
\item there exists $\delta>0$ such that $F+nh\sbq U$ for all $h\in(-\delta,\delta)$;
\item $\lambda(U\setminus F)<\ve/n$.
\end{itemize}

Now, for each $h\in (-\delta,\delta)$ we have
\begin{align*}
\lambda\left(A\setminus\bigcap_{k=1}^n(A+kh)\right) & \le\sum_{k=1}^n\lambda\left(A\setminus(A+kh)\right)
\le\sum_{k=1}^n\lambda\left(U\setminus(F+kh)\right)\\
& =\sum_{k=1}^n\left(\lambda(U)-\lambda(F+kh)\right)<n(\ve/n)=\ve ,
\end{align*}
as desired.
\end{proof}

\begin{corollary} \label{cor}
Let $A\sbq\R$ be a bounded measurable set of positive measure, and let $n$ be a positive integer.
Then there is a neighbourhood $V$ of zero such that for every $h\in V$ there is $x\in A$ such that
$$x+kh\in A\quad\text{for }\quad k=1,2,\dots ,n.$$
\end{corollary}
\begin{proof}
Since $\lambda(A)>0$ is finite, we can use Fact \ref{fa} and pick sufficiently small neighbourhood 
$V=(-\delta,\delta)$ of zero such that $\lambda\left(A\cap\bigcap_{k=1}^n(A+kh)\right)>0$ for all $h\in V$. 
Clearly, we may replace $h$ by $-h$ and pick $x\in A\cap\bigcap_{k=1}^n(A-kh)$ which yields the assertion.
\end{proof}
 
Note that a similar fact and a corollary are true in $\R^N$. In fact,
Corollary \ref{cor} is a measure counterpart of \cite[Lemma 3.3]{BM} which gave in \cite{BM} condition (ii) of Theorem \ref{TT1}. Below, we will use a similar argument to obtain condition (i) of this theorem.
\vspace{1 em}

\noindent
{\bf Proof of Theorem \ref{TT1}(i)}.
Suppose that $f$ is not identically zero. Consider its representation (\ref{EE2}) with $f_p$ that is not identically zero. By our assumption, $f$ vanishes on a measurable set $E$ of positive measure. We may assume that $E$ is bounded. For $A:=E$ and $n:=p$ pick a neighbourhood $V$ of zero as in Corollary \ref{cor}. 
Then for any $h\in V$ we can pick the respective point $x$. Using properties (\ref{EE1}), (\ref{EE3}) and \cite[Corollary 15.1.2]{Ku}, we have
$$f_p(h)=\frac{\Delta_h^p f_p(x)}{p\, !}=\frac{\Delta_h^pf(x)}{p\,!}=\frac{1}{p\,! }
\sum_{k=0}^p {p\choose k}(-1)^{p-k}f(x+kh)=0,$$
where the last equality follows from the property of $V$.
Hence  $f_p$ vanishes on $V$, and by the equality $f_p(ru)=r^p f_p(u)$ for all $r\in\mathbb Q$ and $u\in V$, we infer that $f_p$ vanishes everywhere. This yields a contradiction. \hfill $\Box$

\begin{remark} \label{Rem}
{\em Note that Theorem \ref{alt} is a strengthening of Theorem \ref{TT1}. Indeed, if $f$ vanishes on a set $E\subseteq\R$ of positive measure, then $f(x)f(y)=0$ for all $(x,y)\in E\times\R$ and $E\times\R$ is
of positive plane measure. So, $f$ vanishes everywhere by Theorem \ref{alt}. Similarly, we proceed in the case of the Baire category.}
\end{remark}

Now, we are ready to extend the results on the alternate equation presented in Theorem \ref{alt} to the new cases
where the largeness of $D$ means that it is a Borel set which does not belong to $\mc M\otimes\mc N$ or it does not belong to $\mc N\otimes\mc M$. Our proof is based on the same idea which was used in \cite[Theorem 3.2]{BM}.

\begin{theorem} \label{twu}
Let $f\colon\R\to\mathbb C$ be a generalized polynomial such that $f(x)f(y)=0$ for all $(x,y)\in D$ where $D\sbq\R^2$ is a Borel set such that either $D\notin \mc M\otimes \mc N$ or $D\notin\mc N\otimes \mc M$. Then $f(x)=0$ for each $x\in\R$.
\end{theorem}
\begin{proof}
Assume that $D\notin \mc M\otimes \mc N$. The proof for $D\notin\mc N\otimes \mc M$ is analogous.
Let $P:=\{ x\in\R\colon f(x)=0\}$. We will show that there exists a Borel set $A\sbq P$ such that either $A\notin\mc N$ or $A\notin\mc M$. Then by Theorem \ref{TT1} we obtain the assertion.

By the assumption, we clearly have
\begin{equation} \label{eq}
D\sbq (P\times\R)\cup(\R\times P).
\end{equation}
Since $D\notin \mc M\otimes \mc N$, we obtain
$$B:=\{ x\in\R\colon D_x\notin\mc N\}\notin\mc M.$$
We know that $B$ is Borel; see \cite[22.22]{Ke}. (If we replace $\mc M$ by $\mc N$ in the definition of $B$, then $B$ is again Borel; see \cite[22.25]{Ke}.) 

Consider two cases.

If $D_x\sbq P$ for some $x\in B$, then $A:=D_x\notin\mc N$ is the desired set.

If $D_x\setminus P\neq\emptyset$ for all $x\in B$, fix any $x\in B$. Then for each $y\in D_x\setminus P$ we have $(x,y)\in D\setminus (\R\times P)$, so from (\ref{eq}) it follows that $(x,y)\in P\times\R$. Hence $x\in B$ implies $x\in P$. So, we have proved that $B\sbq P$. Since $B\notin \mc M$, $A:=B$ is the desired set.
\end{proof}

Notice that the respective versions of Mendez ideals can be considered in $\R^k\times\R^k$ with the analogous properties. So, Theorem \ref{twu} can be generalized to the case when $f\colon\R^k\to\mathbb C$.

\section{A theorem on the conditional inequality}
Here we are going to extend the results of Boros and Menzer from Theorem \ref{cond} concerning the conditional inequality
\[ f(x)f(y)\ge 0\]
for a generalized monomial $f\colon\R\to\R$.
We need the following two lemmas.

\begin{lemma}(See \cite[Lemma 3.5]{BM1}, \cite[Lemma 4.4]{BM}.) \label{LL1}
Let $\I:=\mc N$ or $\I:=\mc M$. Assume that $P\sbq\R$ fulfills the conditions:
\begin{itemize}
\item[(i)] for all $r\in\Q$ and $x\in P$ we have $rx\in P$;
\item[(ii)] there is a Borel set $A\sbq P$ with $A\notin\I$.
\end{itemize}
Then $\R\setminus P\in\I$.
\end{lemma}

\begin{lemma}(See \cite[Lemma 4.1]{BM1}, \cite[Lemma 4.5]{BM}.) \label{LL2}
Let $\I:=\mc N$ or $\I:=\mc M$. Assume that $f\colon \R\to\R$ is a generalized monomial of degree $m\ge 1$ 
and $H\sbq\R$ is a closed set such that $\{ x\in\R\colon f(x)\notin H\}\in\I$. Then $f(x)\in H$ for each $x\in\R$.
\end{lemma}

Below, we will mimic ideas of the proof in \cite[Theorem 4.3]{BM}.

\begin{theorem} \label{tww}
Let $f\colon\R\to\R$ be a generalized monomial of an even degree $m$, fulfilling $f(x)f(y)\ge 0$
for all $(x,y)\in D$ where $D\sbq\R^2$ is a Borel set such that either 
$D\notin \mc M\otimes\mc N$ or $D\notin \mc N\otimes\mc M$. Then $f(x)f(y)\ge 0$ holds for all $(x,y)\in\R^2$.
\end{theorem}
\begin{proof}
If $m=0$, the assertion is clear. So, let $m>0$. Define
$$P_f:=\{x\in\R\colon f(x)\ge 0\},\quad N_f:=\{x\in\R\colon f(x)\le 0\} .$$
Since $f(rx)=r^m f(x)$ for all $x\in\R$ and $r\in\mathbb Q$, the sets $P_f$ and $N_f$ satisfy condition (i) in Lemma \ref{LL1}. Also, the set $P_f\cap N_f=\{x\colon f(x)=0\}$ satisfies it. We will check condition (ii) of this lemma for the above sets in the respective situations. 

Assume that $D\notin\mc M\otimes\mc N$. The proof for $D\notin\mc N\otimes \mc M$ is analogous. By the assumption we have
$$D\sbq (P_f\times P_f)\cup(N_f\times N_f).$$
Since $D\notin\mc M\otimes\mc N$, we obtain
$$B:=\{x\in\R\colon D_x\notin\mc N\}\notin\mc M.$$
Like in the proof of Theorem \ref{twu}, we observe that the set $B$ is Borel. Consider three cases.

If $D_x\sbq P_f$ for some $x\in B$, we can check that condition (ii) of Lemma \ref{LL1} holds for $P_f$ with $A:=D_x$. Hence by this lemma $\R\setminus P_f\in\mc N$. Then observe that $f$ and $P_f$ satisfy the assumptions of Lemma \ref{LL2} with 
$H:=[0,\infty)$. Thus $f(x)\ge 0$ for each $x\in\R$.

If $D_x\sbq N_f$ for some $x\in B$, we can check that condition (ii) of Lemma \ref{LL1} holds for $N_f$ with $A:=D_x$. Hence by this lemma $\R\setminus N_f\in\mc N$. Then observe that $f$ and $N_f$ satisfy the assumptions of Lemma \ref{LL2} with 
$H:=(-\infty,0]$. Thus $f(x)\le 0$ for each $x\in\R$.

If $D_x\setminus P_f\neq\emptyset$ and $D_x\setminus N_f\neq\emptyset$ for each $x\in B$, for any $x\in B$ we can pick 
$y_i\in D_x$ ($i=1,2$) such that $f(y_1)< 0$ and $f(y_2)> 0$. Since $(x,y_i)\in D$, we have 
$f(x)f(y_i)\ge 0$ for $i=1,2$. Consequently, $f(x)=0$. Thus we have proved that $B\sbq P_f\cap N_f$.
Hence condition (ii) in Lemma \ref{LL1} holds for $Z_f:=P_f\cap N_f$ and $A:=B$. Now, by Lemma \ref{LL1},
we obtain $\R\setminus Z_f\in\mc M$. Finally, we apply Lemma \ref{LL2} to $f$ and $Z_f$ with $H:=\{0\}$.
So, $f(x)=0$ for each $x\in\R$.

Clearly, the assertion of  the theorem holds in each of the above cases.
\end{proof}

Like it was mentioned for Theorem \ref{twu}, also Theorem \ref{tww} can be generalized to the case when $f\colon\R^k\to\R$.

\section{Concluding remarks}
Although the proofs of Theorems \ref{twu} and \ref{tww} follow several nice ideas from \cite{BM} and \cite{BM1}, they produce new results. There is also a small additional effect. 
Our simple observation that it suffices to consider a 
large set $D\sbq\R^2$ which is Borel (instead of being measurable) can slightly simplify a reasining in the proofs of the respective theorems in \cite{BM}. Indeed, one can omit the use of integrals in the proofs
of \cite[Theorems 3.2 and 4.3]{BM}, and the use of expression $D=U\bigtriangleup T$ in the proofs of 
\cite[Theorems 3.6 and 4.8]{BM} where $U$ is open and $T$ is meager. To this aim it suffices to modify respectively our reasoning in the proofs of Theorems \ref{twu} and \ref{tww}. Here an important role is played by the Fubini and the Kuratowski-Ulam theorems and their converses, and by the fact that the sets 
$\{x\colon D_x\notin\mc M\}$ and $\{x\colon D_x\notin\mc N\}$ are Borel provided that $D\sbq\R^2$ is Borel.

Finally, let us summarize our results and those of \cite{BM}. We formulate it in two corollaries.
For $f\colon\R\to\mathbb C$, we denote $\wf\colon\R^2\to\mathbb C$ by $\wf(x,y):=f(x)f(y)$.

\begin{corollary} \label{C1}
Assume that $f\colon\R\to\mathbb C$ is a generalized polynomial. The following are equivalent:
\begin{itemize}
\item[(1)] $f$ vanishes on a Borel set of positive measure;
\item[(2)] $f$ vanishes on a Borel nonmeager set;
\item[(3)] $\wf$ vanishes on a plane Borel set of positive measure;
\item[(4)] $\wf$ vanishes on a nonmeager plane Borel set;
\item[(5)] $\wf$ vanishes on a Borel set $D\notin\mc M\otimes\mc N$;
\item[(6)] $\wf$ vanishes on a Borel set $D\notin\mc N\otimes\mc M$;
\item[(7)] $f$ vanishes everywhere on $\R$.
\end{itemize}
\end{corollary}

Implications (1) $\to$ (7) and (2) $\to$ (7) are stated in Theorem \ref{TT1}. 

Implications (3) $\to$ (7) and (4) $\to$ (7) are stated in Theorem \ref{alt}; note that Theorem \ref{TT1} is used in the proof.

Implications (5) $\to$ (7) and (6) $\to$ (7) are stated in  Theorem \ref{twu}; note that Theorem \ref{TT1} is used in the proof. 

By Remark \ref{Rem} we have (3) $\to$ (1) and (4) $\to$ (2). Similarly, (6) $\to$ (1) and (5) $\to$ (2). 

Trivially, (7) implies each of the remaining conditions. 

\begin{remark} \label{remi}
{\em Clearly, if $\wf$ nanishes on sets $\{ 0\}\times\R$ or $\R\times\{ 0\}$ that are meager of measure zero, then $f$ vanishes everywhere. It was proved in \cite{BF} that, if $f\colon \R\to\R$ is a generalized polynomial and $\wf$ vanishes on the unit circle, then $f$ vanishes everywhere.}
\end{remark}

The following problem seems interesting.

\begin{problem}
{\em Let us call a set $E\sbq\R$ good if the vanishing of every generalized polynomial on $E$ implies its vanishing everywhere. Is it true that every good set contains a minimal good set?}
\end{problem}

\begin{corollary} \label{C2}
Assume that $f\colon\R\to\R$ is a generalized monomial of even degree. The following are equivalent:
\begin{itemize}
\item[(1)] $\wf\ge 0$ on a plane Borel set of positive measure;
\item[(2)] $\wf\ge 0$ on a nonmeager plane Borel set;
\item[(3)] $\wf\ge 0$ on a Borel set $D\notin\mc M\otimes\mc N$;
\item[(4)] $\wf\ge 0$ on a Borel set $D\notin\mc N\otimes\mc M$;
\item[(5)] $\wf\ge 0$ everywhere on $\R^2$.
\end{itemize}
\end{corollary}

Implications (1) $\to$ (5) and (2) $\to$ (5) are stated in Theorem \ref{cond}. 

Implications (2) $\to$ (5) and (3) $\to$ (5) are stated in Theorem \ref{tww}.

Trivially, condition (5) implies each of the remaining conditions.

\end{document}